\documentclass{amsproc}

\input{macros.tex}

\begin{document}

\setcounter{page}{69}

\title{Fluctuations in first-passage percolation}


\author{Philippe Sosoe}
\address{Center for Mathematical Sciences and Applications, Harvard University, 20 Garden Street, Cambridge, MA 01451, USA\afterpage{\blankpage}}
\curraddr{Cornell University, 584 Malott Hall, Ithaca, NY 14853, USA}
\email{psosoe@math.cornell.edu}


\subjclass[2010]{60K35, 60K37, 82B43}

\date{\today}

\begin{abstract}
We present a survey of techniques to obtain upper bounds for the variance of the passage time in first-passage percolation. The methods discussed are a combination of tools from the theory of concentration of measure, some of which we briefly review. These are combined with variations on an idea of Benjamini-Kalai-Schramm to obtain a logarithmic improvement over the linear bound implied by the Efron-Stein/Poincare inequality, for general edge-weight distributions.
\end{abstract}

\maketitle


\tableofcontents

\section{Introduction}
We look at first-passage percolation  in $\mathbb{Z}^d$, with i.i.d. weights 
\[t_e, \quad e\in \mathcal{E}(\mathbb{Z}^d).\]
Here $\mathcal{E}(\mathbb{Z}^d)$ is the set of edges between points of $\mathbb{Z}^d$ that differ by 1 in one coordinate. The common distribution of the weights is denoted by $\mu$, a probability measure on $[0,\infty)$. We will assume throughout that 
\begin{equation}
\label{eqn: zeroprob}
\mu(\{0\}) < p_c(d),
\end{equation}
where $p_c(d)$ is the critical probability for independent Bernoulli percolation on the edges.

Let $\mathbf{e}_1=(1,\ldots,0)$  be the first standard basis vector.
The quantity we are interested in is the passage time:
\begin{equation}\label{eqn: passagetime}
T_n :=T(0,n\mathbf{e}_1)=\inf_{\gamma: 0\rightarrow n\mathbf{e}_1} \sum_{e\in \gamma}t_e,
\end{equation}
where the infimum is taken over the collection of lattice paths 
\[\gamma=(0=x_0,\ldots, n\mathbf{e}_1=x_N), \quad \|x_i-x_{i+1}\|_1=1.\]

\subsection{First order}
\index{subadditive ergodic theorem}%
 Under our assumptions, the leading order behavior of $T_n$ is linear. This follows from the subadditive ergodic theorem, which is presented in \cite{ch:Damron}. Here we just record the result for future reference.
\begin{theorem}\label{thm: 1storder}
Let the edge weights be i.i.d. with distribution $\mu$ satisfying \eqref{eqn: zeroprob}, as well as
\[\mathbb{E}\min(Y_1,\ldots,Y_{2d}) <\infty,\]
where $Y_i$, $i=1,\ldots, 2d$ are independent and have distribution $\mu$. Then there exists a constant $\nu>0$ such that
\begin{equation}\label{eqn: subadditive-thm}
\lim_{n\rightarrow \infty} \frac{1}{n}T(0,n\mathbf{e}_1)=\nu,
\end{equation}
almost surely and in $L^1$.
\end{theorem}
 
\subsection{Fluctuations}\index{fluctuation}%
 The objective in this article will be to investigate the fluctuations of $T_n$ around its mean $\mathbf{E}T_n$. The simplest measure of the magnitude of these fluctuations is the variance
\[\mathbf{Var}(T_n)= \mathbf{E}(T_n-\mathbf{E}T_n)^2.\]
It is expected that
\begin{equation}\label{eqn: approx}
\mathbf{Var}(T_n) \approx n^{2\chi}
\end{equation}
for some dimension dependent exponent $\chi=\chi(d)$. We leave the exact nature of the approximation \eqref{eqn: approx} unspecified. In some exactly solvable models, the correct value of $\chi$ is known and we have asymptotics of the form
\begin{equation}
\frac{\mathbf{Var}(T_n)}{n^{2\chi}}\rightarrow \sigma,
\end{equation}
for some $\sigma>0$ (see the section on predictions below). We will be interested in upper and lower bounds on the variance.

\subsection{Predictions} 
\index{Kardar-Parisi-Zhang (KPZ)!exponents}%
\index{KPZ exponents}%
The validity of the approximation \eqref{eqn: approx} is widely assumed in the physics literature. Moreover, the exponents $\chi(d)$ are believed to be universal among a large class of growth models, some of which are more complicated than first-passge percolation. The models considered in the physics literature include ballistic aggregation, domain walls in two-dimension random-bond Ising models, and directed polymers in random potentials \cite{HH,KZ}. 

Physicists generally agree that $\chi(d)<1$ and that $\chi$ decreases with $d$, but there are differring predictions for the value of $\chi$ \cite{MW, KZ, KK}. A possibility which is investigated in several works on random growth models is the existence of an upper critical dimension $d_0$, such that $\chi(d)=0$ for $d\ge d_0$. See the discussion in \cite{MW}.

The most precise predictions are available in dimension $d=2$, where it is expected \cite{HH, HHF, KPZ, KZ} that 
\[\chi=\frac{1}{3}.\]
This prediction has been verified rigorously in a number of exactly solvable growth models. We cite one of the early results below. Articles \cite{ch:Seppalainen} and \cite{ch:Corwin} will explore solvable models in greater depth. In any case, these predictions are far from being proved in the first-passage percolation model that concerns us here.

Another important prediction from physics, the universal scaling relation, links the order of the fluctuations $\chi$ to another exponent $\xi$, the transversal fluctuation exponent. Discussing this here would take us too far afield, so we refer to  
\cite{ch:Damron} for a statement of the scaling relation and discussion of its derivation under suitable assumptions.

\subsection{Known results}
The best lower bound for $\chi$ in general dimension is due to H. Kesten in \cite{kesten93}, and is
\begin{equation}\label{eqn: chi-lower}
\mathbf{Var}(T_n) \ge C,
\end{equation}
which translates to
\[\chi(d)\ge 0.\]
Note that if the model has an upper critical dimension in the sense discussed above, or simply $\chi(d)\rightarrow 0$, this is the best that can be hoped for without any assumptions on $d$, although it is likely a poor bound in low dimensions. C. Newman and M. Piza obtain the following improvement on  \eqref{eqn: chi-lower} in dimensions $2$:
\begin{theorem}[Newman, Piza \cite{newmanpiza}]
\index{fluctuation!lower bound}%
Let $d=2$. Suppose $\mathbf{E} t_e^2<\infty$, $\mathbf{Var}(t_e)>0$ and that \eqref{eqn: zeroprob} holds. Then we have the lower bound
\[\mathbf{Var}(T_n)\ge C\log n\]
for some $C>0$.
\end{theorem}
Note that the previous result still does not imply that $\chi(2)>0$.
\index{Benjamini-Kalai-Schramm theorem}%
As for upper bounds, the best known result is the following, obtained in \cite{DHS13} by M. Damron, J. Hanson and myself, following work of I. Benjamin, G. Kalai, and O. Schramm \cite{BKS}, as well as M. Benaim and R. Rossignol \cite{BR}:
\begin{theorem}[Damron, Hanson, Sosoe \cite{DHS13}]\label{thm: DHS13}
Suppose 
\begin{equation}
\label{eqn: logte}
\mathbf{E}\Bigl[t_e^2 \log_2^+ t_e\Bigr] <\infty
\end{equation}
 and \eqref{eqn: zeroprob} holds. Then there is a constant $C>0$ such that
\begin{equation}\label{eqn: DHS13}
\mathbf{Var}(T_n) \le C\frac{n}{\log n}.
\end{equation}
\end{theorem} 
Most of the remainder of this article will be devoted to giving a sketch of the proof of Theorem \ref{thm: DHS13}. Under stronger assumptions on the moments, this bound on the variance can be supplemented with concentration results:
\begin{theorem}[Damron, Hanson, Sosoe \cite{DHS14}]\label{thm: concentration}
Under the same assumptions as Theorem \ref{thm: DHS13}, there exist constants $C,c>0$, such that for every $\lambda \ge 0$:
\begin{equation}
\mathbf{P}(T_n-\mathbf{E}T_n \le -\lambda\sqrt{n/\log n})\le Ce^{-c\lambda}.
\end{equation}
If $\mathbf{E}e^{\alpha t_e}<\infty$, then
\begin{equation}
\mathbf{P}(|T_n-\mathbf{E}T_n|\ge \lambda \sqrt{n/\log n}) \le Ce^{-c\lambda}.
\end{equation}
\end{theorem}
We discuss concentration in the final section.

The upper and lower bounds in \eqref{eqn: chi-lower} and \eqref{eqn: DHS13} imply
\[0\le \chi(d) \le 1/2,\]
in all dimensions, which is quite far from the values found based on arguments from statistical physics and simulations. We close this section by mentioning a remarkable result due to Kurt Johansson in an exactly solvable model of last passage percolation which confirms the predictions of physicists in dimension 2.
\index{Johansson, Kurt}%

In the model considered by Johansson, the weights $t_v$ are placed at the vertices of the square $([0,n]\times [0,n])\cap \mathbb{Z}^2$ in the first quadrant. Allowable paths from $(0,0)$ to $(n,n)$ are constrained to have non-decreasing coordinates, and the \emph{last passage time}
\begin{equation}\label{eqn: last-passage}
T_n = \max_{\pi : (0,0)\rightarrow (n,n)} \sum_{v\in \pi} t_v,
\end{equation}
where the maximum is taken over paths
\[\pi =\{x_0=(0,0),\ldots, x_n=(n,n)\},\]
such that $x_{i+1}= x_i + (1,0)$ or $x_{i+1}=x_i+(0,1)$.
\begin{theorem}[Johansson \cite{johansson2000}]
Let $A$ be an $n\times n$ matrix with complex Gaussian entries of mean zero and variance 1.  Let $\lambda_n$ be the largest eigenvalue of $AA^*$.
If the edge weights have geometric distribution with mean $1$,
\begin{equation}
\mathbf{P}(T_n\le t) = \mathbf{P}(\lambda_n \le t),
\end{equation}
for all $t\ge 0$.
\end{theorem}\index{Tracy-Widom distribution}%
The moments of the quantity
\[Z_n=\frac{\lambda_n-4n}{2^{4/3}n^{1/3}}\]
converges to the GUE \emph{Tracy-Widom distribution}. It follows immediately from this result that $\chi=1/3$.

\section{Kesten's linear bound for the variance}\index{Kesten, Harry}%
In this part of the article, we take a first step towards the derivation of Theorem \ref{thm: DHS13}, by giving a proof of Kesten's result from \cite{kesten93} that the variance is at most linear in any dimension (i.e. $\chi(d)\le 1/2$):
\begin{theorem}[Kesten \cite{kesten93}]\label{thm: var}
Assume $\mathbf{E}t_e^2<\infty$ and that \eqref{eqn: zeroprob} holds. Then
\begin{equation}
\mathbf{Var}(T_n)\le Cn.
\end{equation}
\end{theorem}
Here, as later in this article, we use the symbol $C$ to denote a constant independent of $n$, whose value may change between occurrences. In \cite{kesten93}, Theorem \ref{thm: var} is supplemented by a concentration result on scale $n$ in case $\mu$ has exponential moments. We discuss concentration in the final part of the chaper.

Our main goal in this article is to make explicit the ideas behind the proofs of the results. We give proofs when they are deemed instructive, but we do not always provide all the details. The reader can find these in the relevant papers. In particular, in several places below we deal with infinite sums whose convergence can readily be justified by suitable approximations, without giving the details of such justifications.

A common theme in the derivation of the theorem and the refinements leading to \eqref{eqn: DHS13} is the application of standard tools in concentration inequalities to the random variable $T_n$, viewed as a function of the independent edge weight variables $t_e$, combined with some simple inputs from the model. We summarize the two key properties of the model we use in the next section. We will see that we use only a limited amount of  information about the model. This makes it obvious that the techniques we present here are unlikely to get us close to the values of the exponents predicted by physicists.

For Theorem \ref{thm: var}, the general concentration tool we use and apply to $T_n$ is the \emph{Efron-Stein inequality}:
\index{Efron-Stein inequality}%
\begin{lemma}\label{lem: ES}
Let $f:\mathbb{R}^\mathbb{N} \rightarrow \mathbb{R}$ be a measurable function, and $X_1, X_2, \ldots$ be a collection of independent random variable. Let also $X_i'$, $i\ge 0$, be an independent copy of $X_i$. Assume the random variable $f(X_1,X_2, \ldots)$ has finite second moment.
Then
\begin{equation}
\label{eqn: ES}
\mathbf{Var}(f(X_1,X_2,\ldots)) \le \frac{1}{2}\sum_{i=1}^\infty \mathbf{E}\Bigl[f(X_1,\ldots,X_i, \ldots)-f(X_1, \ldots, X_i', \ldots)\Bigr]^2.
\end{equation}

\end{lemma}
\begin{proof}
Since there is no risk of confusion, we will denote the random variable
\[f(X_1,X_2,\ldots)\]
simply by $f$.
We define a sequence of $\sigma$-algebras by
\begin{align*}
\mathcal{F}_0&=\varnothing\\
\mathcal{F}_i&= \sigma(X_1,\ldots,X_i),
\end{align*}
for $i\ge 0$. Consider the martingale increments
\[\Delta_i = \mathbf{E}[f\mid \mathcal{F}_i]-\mathbf{E}[f\mid \mathcal{F}_{i-1}].\]
Then, by orthogonality of the $\Delta_i$, we have
\begin{equation}\label{eqn: martingale}
\mathbf{E}[\mathbf{E}[f\mid \mathcal{F}_n]^2] - \mathbf{E}[f]^2 = \mathbf{E}[\mathbf{E}[f\mid \mathcal{F}_n]^2] - \mathbf{E}[f\mid \mathcal{F}_0]^2=\sum_{i=1}^n \mathbf{E}[\Delta_i^2].
\end{equation}
Denote by $\mathbf{E}_i$ integration over $X_i$:
\[\mathbf{E}_i f = \int f(X_1, \ldots, x_i, \ldots) \,\mathbf{P}_{X_i}(\mathrm{d} x_i).\]

Here we have denoted by $\mathbf{P}_{X_i}$ the distribution of $X_i$ on $\mathbb{R}$.
Notice that
\[ \Delta_i= \mathbf{E}[f\mid \mathcal{F}_i]-\mathbf{E}[f\mid \mathcal{F}_{i-1}]= \mathbf{E}[f-\mathbf{E}_i[f]\mid \mathcal{F}_i].\]
By Jensen's inequality, we have
\[\mathbf{E}\Delta_i^2 = \mathbf{E}[\abs{\mathbf{E}[f-\mathbf{E}_i[f]\mid \mathcal{F}_i]}^2]\le \mathbf{E}[\abs{f-\mathbf{E}_i f}^2]= \mathbf{E}\mathbf{Var}_i (f).\]
where 
\[\mathbf{Var}_i (f) = \int f^2(X_1,\ldots, x_i, \ldots)\,\mathbf{P}_{X_i}(\mathrm{d}x_i)-\left(\int f(X_1,\ldots, x_i, \ldots)\,\mathbf{P}_{X_i}(\mathrm{d}x_i)\right)^2.\]
We can rewrite the latter quantity as
\[\mathbf{Var}_i (f) = \frac{1}{2}\int\int (f(X_1,\ldots, x_i, \ldots) - f(X_1,\ldots,x_i', \ldots))^2\,\mathbf{P}_{X_i}(\mathrm{d}x_i)\mathbf{P}_{X_i'}(\mathrm{d}x_i').\]

Returning to \eqref{eqn: martingale}, this gives
\begin{align*}
\mathbf{E}[\mathbf{E}[f\mid \mathcal{F}_n]^2] - \mathbf{E}[f]^2 &\le \frac{1}{2}\sum_{i=1}^n \mathbf{E}\Bigl[f(X_1,\ldots,X_i, \ldots)-f(X_1, \ldots, X_i', \ldots)\Bigr]^2\\
&\le \frac{1}{2}\sum_{i=1}^\infty \mathbf{E}\Bigl[f(X_1,\ldots,X_i, \ldots)-f(X_1, \ldots, X_i', \ldots)\Bigr]^2.
\end{align*}
By the assumption of finite second moment, the Doob martingale $\mathbf{E}[f\mid \mathcal{F}_n]$ converges to $f$ in $L^2$, and we obtain the result.
\end{proof}

\subsection{Two ingredients from FPP}\label{sec: ingredients} The two pieces of information about first passage percolation we use to pass from Lemma \ref{lem: ES} to Theorem \ref{thm: var} are
\begin{enumerate}
\item A linearization of the passage time. That is, we need to know how the random variable $T_n$ changes when we modify the value of a single edge, conditional on all the others. This means an estimate for
\[T_n(t, (t_{e'})_{e'\neq e})-T_n(s,(t_{e'})_{e'\neq e}).\]
\item An estimate showing that \emph{geodesic paths} in first-passage percolation have typical size of order $n$. Geodesic paths are minimizing paths $\gamma$ in \eqref{eqn: passagetime}. Under our assumption \ref{eqn: zeroprob} \cite{kestenaspects, 50years}, such minimizing paths exist almost surely, although they may not be unique if $\mu$ has atoms.
\index{geodesic}%
\end{enumerate}

To state more precisely the results we will need, define $G_n$ to be the intersection of all geodesics from the origin of $\mathbb{Z}^d$ to $n\mathbf{e}_1$. Then, concerning the point 2. above, we have
\begin{lemma}[Kesten] \label{lem: kesten-size}
There is a constant $C$ such that, for all $n$
\begin{equation}\label{eqn: Gn-bound}
\mathbf{E}[\#G_n^2] \le C\mathbf{E}T_n^2,
\end{equation}
where $\#G_n$ denotes the cardinality of $G_n$.
\end{lemma}
To understand why this implies that geodesics are one dimensional, note that by simply considering the path
\[(0,\ldots, 0), \mathbf{e}_1, 2\mathbf{e}_1,\ldots, n\mathbf{e}_1,\]
we have
\[\mathbf{E}T_n^2 \le n \mathbf{E}t_e^2 + n(n-1)(\mathbf{E}t_e)^2\le Cn^2.\]
With \eqref{eqn: Gn-bound}, this gives
\begin{equation}\label{eqn: Gn-lin}
\mathbf{E}\#G_n \le (\mathbf{E}(\#G_n)^2)^{1/2}\le Cn.
\end{equation}
Note that \eqref{eqn: Gn-bound} is easily seen to be true if the edge weight distribution is supported on $[a,\infty)$ for some $a>0$. Indeed, in that case we always have
\[\#G_n \le \sum_{e\in \gamma} t_e \le \frac{1}{a}T_n,\]
where $\gamma$ is any minimizing path. This simple argument in the case when the distribution is supported on values bounded below suggests that to prove Lemma \ref{lem: kesten-size}, we must control the proportion of very small edge weights along geodesics. That is the purpose of the following result.
\begin{lemma}[Kesten]\label{lem: perc-lem}
For a lattice path $\gamma$, we define
\[T(\gamma) =\sum_{e\in \gamma} t_e.\]
There exist $a$, $c$ such that for all $n\ge 0$,
\[\mathbf{P}(\exists \text{ self-avoiding } \gamma \text{ containing } 0 \text{ with } \#\ge n \text{ but } T(\gamma)< an)\le e^{-cn}.\]
\end{lemma}
Given the lemma, we obtain Lemma \ref{lem: kesten-size} by considering
\[Z_n = (\#G_n) \mathbf{1}_{T_n < a\#G_n}.\]
Using Lemma \ref{lem: perc-lem}, we have
\[\mathbf{P}(Z_n \ge n) \le e^{-Cn},\]
so
\[\mathbf{E}[\#G_n^2]\le a^{-2}\mathbf{E}T_n^2+\mathbf{E}Z_n^2\le C\mathbf{E}T_n^2.\]

For the first ``ingredient'' above, we have the following concerning the effect of changing one edge weight:
\begin{lemma}\label{linear}
If $t\ge s$, then
\[0\le T_n(t, (t_{e'})_{e'\neq e})-T_n(s,(t_{e'})_{e'\neq e}) \le (t-s) \mathbf{1}_{e\in G_n}(s,(t_{e'})_{e'\neq e}),\]
where the indicator function is 1 on the event if the edge $e$ is in the intersection of all geodesics from $0$ to $n\mathbf{e}_1$ in the configuration where the edge $e$ has weight $s$.
\end{lemma}
\begin{proof}
Unless the edge $e$ is in the intersection of all geodesic paths when $t_e=s$, raising the weight of that edge has no effect on the passage time. If $e\in G_n$ for $t_e=s$, then the passage time changes at most linearly in $t_e$.
\end{proof}

\subsection{Finishing the proof of Theorem \ref{thm: var}}
Applying Lemma \ref{linear} with $f=T_n$ and $X_i=t_{e_i}$ for some enumeration
\[e_1,e_2,\ldots\]
of the edges, we see that the quantity to control is
\[\sum_{e\in\mathcal{E}(\mathbb{Z}^d)} \mathbf{E}\Bigl[T_n(t_e,(t_{f})_{f\neq e})- T_n(t_e',(t_{f})_{f\neq e})\Bigr]^2.\]
To simplify notation, we represent the vector $(t_{f})_{f\neq e}$ by $t_{e^c}$.
Since $t_e$ and $t_{e'}$ have the same distribution, we have
\[\mathbf{E}[T_n(t_e,t_{e^c})-T_n(t_e',t_{e^c})]^2 =2\mathbf{E}\Bigl[(T_n(t_e,t_{e^c})-T_n(t_e',t_{e^c}))^2\mathbf{1}_{t_e'>t_e}\Bigr].\]
The insertion of the indicator function $\mathbf{1}_{t_e>t_e'}$ allows us to apply Lemma \ref{linear} to bound
\begin{equation}\label{eqn: varbound}
\begin{split}
\mathbf{E}[(T_n(t_e,t_{e^c})-T_n(t_e',t_{e^c}))^2\mathbf{1}_{t_e>t_e'}]&\le \mathbf{E}[(t_e'-t_e)_+^2 \mathbf{1}_{e\in G_n(t_e,t_e^c)}]\\
&\le \mathbf{E}[(t_e')^2 \mathbf{1}_{e\in G_n(t_e,t_e^c)}]\\
&= \mathbf{E}(t_e')^2 \mathbf{P}( e\in G_n).
\end{split}
\end{equation}
In the final step, we used independence between $t_e'$ and $(t_e,t_e^c)$. Since $\mathbf{E}t_e^2 <\infty$, we have now shown that
\[\mathbf{Var}(T_n) \le C\sum_{e\in \mathcal{E}(\mathbb{Z}^d)} \mathbf{P}(e\in G_n) = C\mathbf{E}[\# G_n].\]
We can now use our second ingredient, the linear bound for the size of $G_n$, to conclude:
\begin{equation}
\label{eqn: kestenvarconc}
\mathbf{Var}(T_n) \le C\mathbf{E}\#G_n \le Cn.
\end{equation}

\subsection{Summary}
Let us summarize the steps that led us to Theorem \ref{thm: var}. The improvement \eqref{eqn: DHS13} will be obtained by following the same strategy, replacing each step with a slight refinement.
\begin{enumerate}
\item We used a martingale decomposition to isolate the contribution of the individual edges. (The Efron-Stein inequality, Lemma \ref{lem: ES}).
\item Using the linearization in Lemma \ref{linear} (``Ingredient 1''), we find that the variance $\mathbf{Var}(T_n)$ is bounded by a sum, where each summand has the form $\mathbf{E}[X\mathbf{1}_{e\in G_n}]$, where $X$ can be decorrelated from the indicator.
\item The sum over $e$ of $\mathbf{P}(e\in G_n)$ contributes $Cn$ (``Ingredient 2'').
\end{enumerate}

\section{The Benjamini-Kalai-Schramm strategy}
In their paper \cite{BKS}, Benjamini, Kalai and Schramm (BKS) showed how to obtain the sublinear bound \eqref{eqn: DHS13} for an edge weight distribution of Bernoulli type
\begin{equation}\label{eqn: mubernoulli}
\mu = \frac{1}{2}\delta_a+\frac{1}{2}\delta_b,
\end{equation}
where $0<a<b$.
\begin{theorem}[Benjamini, Kalai, Schramm]\label{thm: BKS}
If $\mu$ has the form \eqref{eqn: mubernoulli}, then there is a constant $C>0$ such that 
\begin{equation}\label{eqn: BKSbound}
\mathbf{Var}(T_n) \le C\frac{n}{\log n}.
\end{equation}
\end{theorem}
Benjamini-Kalai-Schramm obtained this result by replacing the Efron-Stein inequality \ref{eqn: ES} by the following result of M. Talagrand \cite{talagrand-russo}:\index{Talagrand, Michel}%
\begin{theorem}\label{thm: talvar}
Consider a function $f:\{-1,1\}^n\rightarrow \mathbb{R}$.
Let $X=(X_1,\ldots,X_n)$ be uniformly distributed over the hypercube $\{-1,1\}^n$. 
Let 
\[X^i=(X_1,\ldots, -X_i, \ldots X_n).\]
Then, for some constant $C>0$:
\begin{equation}\label{eqn: talvar}
\mathbf{Var}(f) \le C\sum_{i=1}^n \frac{\mathbf{E}[f(X)-f(X^i)]^2}{1+\log \frac{\|f(X)-f(X^i)\|_2}{\|f(X)-f(X^i)\|_1}}.
\end{equation}
The $L^1$ and $L^2$ norms in \eqref{eqn: talvar} are taken with respect to the uniform measure on $\{-1,1\}^n$.
\end{theorem}
Note that without the logarithm term in the denominator, \eqref{eqn: talvar} is the same as the right side of the of the Efron-Stein inequality \eqref{eqn: ES}. Indeed, if $g:\{-1,1\}\rightarrow \mathbb{R}$, and $x, x'$ are uniformly distributed on $\{-1,1\}$, then
\begin{equation}\label{eqn: discrete-deriv}
\mathbf{E}|g(x)-g(-x)|^2= \mathbf{E}|g(-1)-g(1)|^2=2\mathbf{E}|g(x)-g(x')|^2.
\end{equation}

We will not give the proof of Theorem \ref{thm: talvar}, referring instead to \cite{talagrand-russo, BKS}. The inequality is  specific to the uniform distribution on the hypercube. We derive a related inequality below when we discuss the proof of \eqref{eqn: DHS13} for general edge weights.

\subsection{Small influences} \index{influence}
\[\mathbf{E}\# G_n \le Cn.\]
In dimensions $d\ge 2$, this means geodesics are a ``thin'' set compared to the ambient space. Writing $\# G_n = \sum_e \mathbf{1}_{e\in G_n}$ and exchanging $\mathbf{E}$ and the sum over $e$, \eqref{eqn: Gn-lin} implies
\[\sum_{e\in \mathcal{E}(\mathbb{Z}^2)} \mathbf{P}(e\in G_n)\le Cn.\]
Because of \eqref{lem: kesten-size}, we expect about $n^d$ edges to contribute to the sum on the left. This suggests that, on average, $\mathbf{P}(e\in G_n)$ is of order $n^{1-d}$: each individual has only a small probability of having any effect on the passage time $T_n$. In this case, we say the corresponding variable $t_e$ has \emph{small influence} \cite{garban}.

\subsection{FPP on the torus} C. Garban and J. Steif introduced a model where it is easy to show that edges have small influences in a strong sense. Instead of $\mathbb{Z}^d$, consider the torus $\mathbb{T}^d_n=(\mathbb{Z}/n\mathbb{Z})^d$, and the passage time
\begin{equation}\label{eqn: torus}
T_{\mathbb{T}^d_n}=\inf_{\gamma \in C_m}\sum_{e\in \gamma}t_e,
\end{equation}
where $C_m$ is the set of closed lattice paths in $\mathbb{T}^d_n$ that have winding number one in the first coordinate direction.

The advantage of this model is that, because of the symmetries of the torus, the
probability that an edge $e\in \mathbb{T}^d_n$ is in the intersection $G_n(\mathbb{T}_n^d)$ of all minimizing paths in \eqref{eqn: torus} is independent of $e$:
\begin{equation}\label{eqn: eeprime}
\mathbf{P}(e\in G_n(\mathbb{T}_n^d))=\mathbf{P}(e'\in G_n(\mathbb{T}^d_n))
\end{equation}
for all $e,e'\in \mathbb{T}_n^d$. On the other hand, if the edge weight distribution $\mu$ has the form \eqref{eqn: mubernoulli}, then almost surely,
\[ \#G_n(\mathbb{T}_n^d) \le \frac{1}{a}T_{\mathbb{T}^d_n} \le \frac{b}{a}n.\]
Together with \eqref{eqn: eeprime}, this gives
\begin{align*}
n^d \mathbf{P}(e\in G_n(\mathbb{T}_n^d))  &= \sum_{e'\in \mathcal{E}(\mathbb{T}_n^d)}\mathbf{P}(e'\in G_n(\mathbb{T}_n^d)) \\
&= \mathbf{E}\#G_n(\mathbb{T}_n^d)\\
&\le \frac{b}{a}n.
\end{align*}
From this we obtain the influence bound:
\begin{equation}\label{eqn: lowinfluence}
\mathbf{P}(e\in G_n(\mathbb{T}_n^d))\le Cn^{-c},
\end{equation}
where $c>0$ if $d\ge 2$, uniformly in $e$.

Let us see how \eqref{eqn: lowinfluence} allows us to apply \eqref{thm: talvar} to prove the analog of the Benjamini-Kalai-Schramm result, Theorem \ref{thm: BKS} on $\mathbb{T}_n^d$. We take $f(X)=T_{\mathbb{T}_n^d(X)}$, where
\[t_{e_i} =  \frac{X_i-1}{2}a + \frac{X_i+1}{2}\cdot b, \]
for some enumeration $e_1,e_2,\ldots, e_{n^d}$ of the edges of $\mathbb{T}^d_n$.

To exploit the small influence bound, note that
\begin{equation} \label{eqn: diff}
|T_{\mathbb{T}_n^d}(X)-T_{\mathbb{T}_n^d}(X^i)|\le |T_{\mathbb{T}_n^d}(X)-T_{\mathbb{T}_n^d}(X^i)| \mathbf{1}_{e_i\in G_n(\mathbb{T}_n^d)},
\end{equation}
because the left side is only non-zero if the edge $e_i$ is in $G_n(\mathbb{T}_n^d)$. By Cauchy-Schwarz, \eqref{eqn: diff}, \eqref{eqn: lowinfluence}:
\begin{align*}
\|T_{\mathbb{T}_n^d}(X)-T_{\mathbb{T}_n^d}(X^i)\|_1 &\le \mathbf{P}(e\in G_n(\mathbb{T}_n^d))^{1/2}\|T_{\mathbb{T}_n^d}(X)-T_{\mathbb{T}_n^d}(X^i)\|_2\\
&\le Cn^{-c/2}\|T_{\mathbb{T}_n^d}(X)-T_{\mathbb{T}_n^d}(X^i)\|_2.
\end{align*}
Thus, for $i=1,\ldots,n^d$:
\[\frac{\|T_{\mathbb{T}_n^d}(X)-T_{\mathbb{T}_n^d}(X^i)\|_2}{\|T_{\mathbb{T}_n^d}(X)-T_{\mathbb{T}_n^d}(X^i)\|_1}\ge Cn^{c/2},\]
for some constant $C$.
Inserting this into \eqref{eqn: talvar} and using $\mathbf{E}\# G_n(\mathbb{T}_n^d)\le Cn$, we find
\begin{align*}
\mathbf{Var}(T_{\mathbb{T}^d_n})&\le \frac{C}{\log n}\sum_{e\in \mathcal{E}(\mathbb{T}_n^d)}\mathbf{P}(e\in G_n(\mathbb{T}_n^d))\\
&\le C\frac{n}{\log n}.
\end{align*}
which concludes the proof of sublinear variance on the torus.

\subsection{Averaging}
The difficulty in reproducing the argument in the previous section for the passage time $T_n$ in $\mathbb{Z}^d$ is that no inequality like \eqref{eqn: lowinfluence} can be true of edges that are close to the origin or $n\mathbf{e}_1$. For instance, if $\mathbf{e}_i=(0,\ldots, 1,\ldots,0)$, $i=1,\ldots, d$ are the standard basis vectors, then
\[1\le \sum_{i=1}^{d}( \mathbf{P}(\mathbf{e}_i \in G_n)+ \mathbf{P}(-\mathbf{e}_i \in G_n)),\]
so one of the $2d$ probabilities on the right is bounded below by $1/2d$.

Benjamini-Kalai-Schramm resolved this difficulty by considering an averaged version of the passage time for which the edge variables do have small influence. Here we will use a slightly different average. Let $T(x,y)$ be the passage time between $x$ and $y$:
\[T(x,y)=\inf_{\gamma: x\rightarrow y} \sum_{e\in \gamma}t_e.\]
We let 
\begin{equation}\label{eqn: mchoice}
m=\lceil n^{1/4}\rceil, 
\end{equation}
and define the averaged passage time
\begin{equation}
\label{eqn: Fn}
F_n = \frac{1}{\# B_m}\sum_{z\in B_m} T(z,z+n\mathbf{e}_1),
\end{equation}
where 
\[B_m=\{x\in \mathbb{Z}^d: \|x\|_1\le m\}.\]

The next approximation result shows that for the purposes of showing \eqref{eqn: DHS13}, we can consider $F_n$:
\begin{lemma}\label{lem: approx}
Suppose $\mathbf{E}t_e^2<\infty$. For $d\ge 2$, there is a constant such that for all $n\ge 1$,
\begin{equation}
|\mathbf{Var}(T_n)-\mathbf{Var}(F_n)|\le Cn^{3/4}.
\end{equation}
\end{lemma}
\begin{proof}
First, write
\[|\mathbf{Var}T(0,n\mathbf{e}_1)-\mathbf{Var}F_n|\le \|{T}(0,n\mathbf{e}_1)-F_n\|_2(\mathbf{Var}^{1/2}(T_n)+\mathbf{Var}^{1/2}(F_n)).\]
By subadditivity,
\[|T(0,n\mathbf{e}_1)-T(z,z+n\mathbf{e}_1)|\le T(0,z)+T(n\mathbf{e}_1,z+n\mathbf{e}_1), \]
so from \eqref{eqn: Fn} and Minkowski's inequality,
\[\|F_n-T(0,n\mathbf{e}_1)\|_2\le \frac{2}{\#B_m}\sum_{\|z\|_1\le m} \|T(0,z)\|_2 \le 2\max_{\|z\|_1\le m}\|T(0,z)\|_2.\]
Since $\mathbf{E}t_e^2<\infty$ and the passage time from the origin to $z$ can be dominated by a sum of $O(\|z\|_1)$ i.i.d. $\mu$-distributed variables, we have
\[\|T(0,z)\|_2\le C\|z\|_1,\]
for some constant $C>0$, so using \eqref{eqn: varbound} on $\mathbf{Var}T_n$, $\mathbf{Var}F_n$, we have
\[|\mathbf{Var}T_n-\mathbf{Var}F_n|\le C\|F_n-T(0,x)\|_2\cdot n^{1/2}\le Cn^{3/4}, \]
by our choice of $m$ \eqref{eqn: mchoice}.
\end{proof}

\subsection{Entropy}\index{entropy}%
To control the averaged passage time \eqref{eqn: Fn}, we will use a substitute for Theorem \ref{thm: talvar}, due to F. Falik and G. Samorodnitsky \cite{FS}, based on an entropy inequality.

Let $X\in L^1(\mu)$ be a nonnegative random variable. The entropy of $X$ relative to $\mu$ is defined by
\begin{equation}
Ent_\mu X = \mathbf{E} \Bigl[X\log \frac{X}{\mathbf{E}X}\Bigr].
\end{equation}
By Jensen's inequality, $Ent_\mu X\ge 0$ for any $X$.
 
 \index{Falik-Samorodnitsky inequality}%
 \index{martingale increments}%
\begin{lemma}[Falik, Samorodnitsky]\label{lem: FS}
Let $(\Omega, \mathbf{P})$ be a probability space and $f \in L^2(\Omega,\mathbf{P})$ be a function of independent random variables $X_1,X_2,\ldots$ . Let 
\[\mathcal{F}_0=\varnothing, \mathcal{F}_i = \sigma(X_1,\ldots, X_i),\] 
and 
\begin{equation}
\label{eqn: Di}
\Delta_i f = \mathbf{E}[f\mid \mathcal{F}_i]-\mathbf{E}[f\mid \mathcal{F}_{i-1}]
\end{equation}
denote the martingale increments. We have the inequality
\begin{equation}\label{eqn: FS}
\mathbf{Var}(f)\cdot\log \frac{\mathbf{Var}(f)}{\sum_{i=1}^\infty (\mathbf{E}|\Delta_i f|)^2 } \le \sum_{i=1}^\infty  Ent_\mu (\Delta_i f)^2
\end{equation}
\end{lemma}

Before giving the proof of Lemma \ref{lem: FS}, we need one more auxiliary result:
\begin{proposition}[Falik-Samorodnitsky]\label{prop: lb}
If $X \geq 0$ almost surely,
\begin{equation}\label{eqn: entlow}
Ent_\mu (X^2) \geq \mathbf{E}_\mu X^2 \log \frac{\mathbf{E}_\mu X^2}{(\mathbf{E}_\mu X)^2}\ .
\end{equation}
\end{proposition}
\begin{proof}
Since both sides are homogeneous under scaling, we assume $\mathbf{E}f^2=1$, so the inequality to prove is
\[-\log(\mathbf{E}f)^2 \le \mathbf{E}f^2\log f^2,\]
or, as $f\ge 0$,
\[0\le \mathbf{E}f^2\log(f\mathbf{E}f).\]
Next, use the elementary inequality
\[1-x\le \log x^{-1}\]
with $x=(f\mathbf{E}f)^{-1}$ to find
\[0=\mathbf{E}f^2-1 = \mathbf{E}\left[f^2(1-1/(f\mathbf{E}f)); f>0\right]\le \mathbf{E}f^2\log(f\mathbf{E}f).\qedhere\]\end{proof}

We can now give the proof of \eqref{eqn: FS}:
\begin{proof}[Proof of Lemma \ref{lem: FS}]
We apply \eqref{eqn: entlow} to each of the $\Delta_i f$:
\begin{align*}
\sum_{i=1}^\infty Ent_\mu (\Delta_i f)^2&\ge \sum_{i=1}^\infty \mathbf{E}[(\Delta_if)^2]\log \frac{\mathbf{E}[(\Delta_i f)^2]}{(\mathbf{E}|\Delta_i f|)^2}\\ 
&= -\mathbf{Var}(f) \sum_{i=1}^\infty \frac{\mathbf{E}(\Delta_i f)^2}{\mathbf{Var}(f)}\log \frac{(\mathbf{E}|\Delta_i f|)^2}{\mathbf{E}[(\Delta_i f)^2]}\,.
\end{align*}
Since $f\in L^2$, $\mathbf{Var}(f) =\sum_{i=1}^\infty \mathbf{E}[(\Delta_i f)^2]$, we apply Jensen's inequality to the function $-\log x$ to obtain the lower bound
\[-\mathbf{Var}(f) \log \sum_{i=1}^\infty \frac{\mathbf{E}[(\Delta_i f)^2]}{\mathbf{Var} (f)}\frac{(\mathbf{E}|\Delta_i f|)^2}{\mathbf{E}[(\Delta_i f)^2]}.\qedhere\]
\end{proof}

\subsection{BKS in $\mathbb{Z}^d$}
Let us see how \eqref{eqn: FS} applied to $f=F_n$ is used to obtain the Benjamini-Kalai-Schramm result in $\mathbb{Z}^d$, and then extended to more general edge weights. The sigma algebra $\mathcal{F}_i$ is $\sigma(t_{e_1},\ldots, t_{e_i})$ for some enumeration of the edges of $\mathbb{Z}^d$. The proof has two essentially distinct parts
\begin{enumerate}
\item Showing that the factor 
\[\log \frac{\mathbf{Var}(f)}{\sum_{i=1}^\infty (\mathbf{E}|\Delta_i f|)^2 }\]
on the left of \eqref{eqn: FS} gives a logarithmic improvement. Here we use a more refined version of our ``Ingredient 2'', the one-dimensionality of geodesics, in the form of Lemma \ref{lem: geo1d}.
\item Proving the bound 
\begin{equation}\label{eqn: entropybd}
\sum_{i=1}^\infty  Ent_\mu |\Delta_i F_n|^2\le Cn.
\end{equation}
\index{logarithmic Sobolev inequality}%
For Bernoulli edge weights, or more generally when a \emph{logarithmic Sobolev inequality} is available, this can be simply estimated by $\sum_{e\in\mathcal{E}(\mathbb{Z}^d)} \mathbf{P}(e\in G_n)$. In the case of general edge weights $\mu$, we represent $\mu$-distributed random variables as suitable images of Bernoulli random variables, which leads to a more complicated calculation presented in the next part of this article.
\end{enumerate}

To see 1. above, note that we may assume 
\begin{equation}
\label{eqn: 78}
\mathbf{Var}(f) \ge n^{7/8},
\end{equation}
otherwise there is nothing left to prove.

Then, write as in the proof of Kesten's linear variance bound 
\begin{equation}\label{eqn: 1dcalc-1}
\begin{split}
\mathbf{E}|\Delta_i F_n| &\le \frac{2}{\#B_m}\sum_{z\in B_m}\mathbf{E}|\Delta_i T(z,z+n\mathbf{e}_1)|\\
&\le \frac{C}{\# B_m}\sum_{z\in B_m}\mathbf{P}(e_i\in G_n(z,z+n\mathbf{e}_1)),
\end{split}
\end{equation}
where $G_n(x,y)$ is the intersection of all geodesics from $x$ to $y$.
By translation invariance, the sum in the last step is
\begin{equation}\label{eqn: 1dcalc-2}
\begin{split}
\sum_{z\in B_m}\mathbf{P}(e_i\in G_n(z,z+n\mathbf{e}_1)) &= \sum_{z\in B_m}\mathbf{P}(e_i+z\in G_n(0,n\mathbf{e}_1))\\
&= \mathbf{E} \#(G_n(0,n\mathbf{e}_1)\cap \{e_i+ B_m\}).
\end{split}
\end{equation}
Since any geodesic is typically a ``1-dimensional'' set, we would expect the last quantity to be of order $\mathrm{diam} B_m$, which is the content of the following 
\begin{lemma}\label{lem: geo1d} Let $\mathcal{G}$ be the set of all finite self-avoiding geodesics from $0$ to $n\mathbf{e}_1$. For any finite set $E$,
\[\mathbf{E}\max_{\gamma \in \mathcal{G}}\#(E\cap \gamma) \le C\mathrm{diam}E.\]
\end{lemma}
We prove Lemma \ref{lem: geo1d} at the end of this section. Assuming the lemma for now, an argument similar to \eqref{eqn: varbound}, but assuming only $\mathbf{E}t_e<\infty$, gives
\begin{equation}\label{eqn: onedelta}
\mathbf{E}|\Delta_i F_n| \le \frac{C}{\# B_m}\mathrm{diam}B_m\le Cn^{(1-d)/4}.
\end{equation}
Using \eqref{eqn: onedelta}, we can bound the sum appearing on the left side in \eqref{eqn: FS}:
\begin{equation}\label{eqn: 54}
\sum_i (\mathbf{E}|\Delta_i F_n|)^2 \le Cn^{(1-d)/4}\sum_i\frac{1}{\# B_m}\sum_{z\in B_m}\mathbf{P}(e_i\in G_n(z,z+n\mathbf{e}_1))\le Cn^{(5-d)/4}.
\end{equation}
Here we have used translation invariance and $\sum_i \mathbf{P}(e_i\in G_n)=\mathbf{E}[\#G_n]$.
From \eqref{eqn: 54}, \eqref{eqn: 78}, we now have, for $d\ge 2$,
\[\mathbf{Var}(F_n) \log n \le C\sum_{i=1}^\infty  Ent_\mu |\Delta_i F_n|^2.\]
Theorem \ref{thm: DHS13} would now be proved if we had the estimate \eqref{eqn: entropybd} for general $\mu$.

\subsection{Log-Sobolev}\index{Bonami-Gross inequality}%
When $\mu$ is of the form \eqref{eqn: mubernoulli}, the left side of \eqref{eqn: entropybd} can be estimated by the right side of \eqref{eqn: ES}. This is due to the following discrete \emph{logarithmic Sobolev inequality}, due to Bonami \cite{bonami} and Gross \cite{gross}:
\begin{proposition}\label{prop: logsobolev}
Let $f:\{0,1\}\rightarrow \mathbb{R}$ and $\mu = \frac{1}{2}(\delta_{0}+\delta_1)$. Then
\begin{equation}\label{eqn: logsobolev}
Ent_\mu f^2 \le (1/2)|f(0)-f(1)|^2.
\end{equation}
\end{proposition}
The proof of this result is elementary and can be found in \cite{BLM}.

At this point we will need the following``tensorization property'' of entropy. 
\begin{proposition}\label{prop: tensorization}
Let $f$ be a non-negative random variable on a product probability space
\[\left(\prod_{i=1}^\infty \Omega_i,\mathcal{F}, \mu = \prod_{i=1}^\infty \mu_i \right),\]
where $\mathcal{F} = \bigvee_{i=1}^\infty \mathcal{G}_i$ and each triple $(\Omega_i,\mathcal{G}_i,\mu_i)$ is a probability space. Then
\begin{equation}
Ent X \leq \sum_{i=1}^\infty \mathbf{E} Ent_{i}X\ ,\label{eqn: tensor}
\end{equation}
where $Ent_i f$ is the entropy of $f(\omega)=f(\omega_1,\ldots, \omega_i, \ldots ))$ with respect to $\mu_i$, as a function of the $i$-th coordinate (with all other values fixed).
\end{proposition}
We refer to \cite[Theorem 2.3]{DHS14} for a proof using the variational characterization of entropy:
\begin{equation}\label{eqn: entropyvar}
Ent f = \sup\{\mathbf{E}fg: \mathbf{E}_\mu e^g\le 1\}.
\end{equation}

Using \eqref{eqn: logsobolev}, we can finish the derivation of Theorem \ref{thm: BKS}. Indeed, we then have, first by Proposition \ref{prop: tensorization}, and then by \eqref{eqn: logsobolev}
\begin{equation}\label{eqn: lssplit}
\begin{split}
\sum_{i=1}^\infty  Ent_\mu |\Delta_i F_n|^2 &\le \sum_{i=1}^\infty \sum_{k=1}^\infty \mathbf{E}Ent_k |\Delta_i F_n|^2\\
&\le C\sum_{i=1}^\infty \sum_{k=1}^\infty \mathbf{E} \Bigl[|\Delta_i F_n(a,t_{e_k^c})-\Delta_i F_n(b,t_{e_k^c})|^2\Bigr]\\
\end{split}
\end{equation}
Now we use
\begin{align}\label{eqn: marta}
\setlength{\nulldelimiterspace}{0pt} 
\begin{split}
&(\Delta_i F_n)(a,t_{e_k^c})-(\Delta_i F_n)(b,t_{e_k^c})\\
&\qquad\quad=
\left\{
\begin{aligned}
&0,&i<k\\
&\mathbf{E}[F_n(a,t_{e_k^c})-F_n(b,t_{e_k^c})\mid \mathcal{F}_i],&i=k\\
&\begin{aligned}
&\mathbf{E}[F_n(a,t_{e_k^c})-F_n(b,t_{e_k^c})\mid \mathcal{F}_i]\\
&\quad\quad-\mathbf{E}[F_n(a,t_{e_k^c})-F_n(b,t_{e_k^c})\mid \mathcal{F}_{i-1}],
\end{aligned}&i>k
\end{aligned}\right.
\end{split}
\end{align}
and orthogonality of martingale differences to find
\begin{align*}
&\sum_{i=1}^\infty \sum_{k=1}^\infty \mathbf{E} \Bigl[|\Delta_i F_n(a,t_{e_k^c})-\Delta_i F_n(b,t_{e_k^c}|^2\Bigr]\\
&\qquad\le \sum_{k=1}^\infty \mathbf{E}\Bigl[|F_n(a,t_{e_k^c})-F_n(b,t_{e_k^c})|^2\Bigr].
\end{align*}
At this point, we can proceed as in \eqref{eqn: varbound}, to estimate
\[\sum_{k=1}^\infty \mathbf{E}\Bigl[|F_n(a,t_{e_k^c})-F_n(b,t_{e_k^c})|^2\Bigr] \le \mathbf{E}[\# G_n(0,n\mathbf{e}1)]\le Cn.\]
This completes our derivation of Theorem \ref{thm: BKS}, except for the proof of Lemma \ref{lem: geo1d}.

Let us note that if instead of \eqref{eqn: mubernoulli}, the edge weight distribution $\mu$ is absolutely continuous and satisfies the logarithmic Sobolev inequality
\begin{equation}\label{eqn: gradls}
Ent_\mu f \le C_{LS}\|f'\|_{L^2(\mu)}^2
\end{equation}
for all smooth $f:\mathbb{R}\rightarrow \mathbb{R}$, then the same argument we have given for Bernoulli also gives Theorem \ref{thm: DHS13} for $\mu$. Indeed, applying \eqref{eqn: gradls} in each variable as in \eqref{eqn: lssplit}, we find
\begin{align*}
\sum_{i=1}^\infty  Ent_\mu |\Delta_i F_n|^2 &\le C\sum_{k=1}^\infty \sum_{i=1}^\infty \mathbf{E}|\partial_{t_{e_k}} (\Delta_iF_n)|^2.\\
&= C\sum_{e \in \mathcal{E}(\mathbb{Z}^d)}^\infty \mathbf{E}|\partial_{t_e} F_n|^2,
\end{align*}
where we have used orthogonality of martingale increments and
\[\partial_{t_{e_k}} (\Delta_iF_n)= \begin{cases} 0, &i<k,\\
\mathbf{E}[\partial_{t_{e_k}}F_n\mid \mathcal{F}_i], &i=k,\\
\mathbf{E}[\partial_{t_{e_k}}F_n\mid \mathcal{F}_i]-\mathbf{E}[\partial_{t_{e_k}}F_n\mid \mathcal{F}_{i-1}],&i>k.
\end{cases}\]
Then, use
\[\partial_{t_{e}} T(0,x)  =\mathbf{1}_{\{e\in G_n(0,x)\}},\]
which holds Lebesgue almost surely (see Lemma \ref{lem: linearb}), and conclude as in \eqref{eqn: kestenvarconc}.

The class of distributions for which \eqref{eqn: gradls} holds includes, for example, any absolutely continuous distribution with bounded density on an interval \cite[Proposition 5.1.6]{bakryanalysis}. The case of general $\mu$ is more delicate, and will be discussed in the next part of this article.

\subsection{Geodesics are one-dimensional}\index{geodesic}%
Recall from \eqref{eqn: 1dcalc-1}-\eqref{eqn: 54} that what allowed us to obtain a logarithmic improvement over the linear bound in Kesten's result Theorem \ref{thm: var} is the ``one-dimensionality'' of geodesics. This is expressed by the estimate in Lemma \ref{lem: geo1d}. This section we show how to obtain this result from Kesten's geodesic length estimate, Lemma \ref{lem: perc-lem}.

\begin{proof}[Proof of Lemma \ref{lem: geo1d}]
Choose $a, c$ from Theorem~\ref{lem: perc-lem}.

If $\# (E \cap \gamma) \geq \lambda$ for some $\gamma \in \mathcal{G}$, then we may find the first and last intersections (say $y$ and $z$ respectively) of $\gamma$ with $V$, the set of endpoints of edges in $E$. The portion of $\gamma$ from $y$ to $z$ is then a geodesic with at least $\lambda$ edges. 
This means
\[
\mathbf{P}(\# (E \cap \gamma) \geq \lambda \text{ for some }\gamma \in \mathcal{G}) \leq (\# V) \exp(-c
\lambda) + \mathbf{P}\left( \max_{y,z \in V}T(y,z) \geq a \lambda \right)\ .
\]
Therefore
\begin{align*}
\mathbf{E} \max_{\gamma \in \mathcal{G}} \#(E \cap \gamma) 
&\leq \text{diam}(E) + \sum_{\lambda=\text{diam}(E)}^\infty (\#V) \exp(-c\lambda)\\
&\qquad+ \sum_{\lambda =  \text{diam}(E) }^\infty \mathbf{P}\left( \max_{y,z \in V}T(y,z) \geq a \lambda \right)\ .
\end{align*}
By the inequality $\text{diam}(E) \geq C(\#V)^{1/d}$, the middle term is bounded uniformly in $E$, so we get the upper bound
\[
C\text{diam}(E) + \frac{1}{a} \mathbf{E} \max_{y,z \in V} T(y,z)\ .
\]
By Lemma~\ref{lem: max_lemma} below, this is bounded by $C\text{diam}(E)$.
\end{proof}

We conclude the proof with the following result, which we used at the end of the previous derivation.
\begin{lemma}\label{lem: max_lemma}
If $\mathbf{E}t_e^2<\infty$, there exists $C$ such that for all finite subsets $S$ of $\mathbb{Z}^d$,
\[
\mathbf{E} \left[\max_{x,y \in S} T(x,y)\right] \leq C\mathrm{diam }~S\ .
\]
\end{lemma}
\begin{proof}
Given $x,y \in S$, we first obtain an estimate for the tail of the distribution of $T(x,y)$. For this, note that we can build $2d$ disjoint (deterministic) paths from $x$ to $y$ of length $N\|x-y\|_1)$ for some integer $N$. This means that $\tau(y,z)$ is bounded above by the minimum of $2d$ variables $T_1, \ldots, T_{2d}$, the collection being i.i.d. and each variable distributed as the sum of $N\text{diam}(S)$ i.i.d. variables $t_e$, so
\[
\mathbf{P}(T(x,y) \geq \lambda) \leq \prod_{i=1}^{2d} \mathbf{P}(T_i \geq \lambda) \leq \left[ \frac{N\text{diam}(S)\mathbf{Var}(t_e)}{(\lambda-N\text{diam}(S) \mathbf{E} t_e)^2} \right]^{2d}\ .
\]
Therefore if we fix some $x_0 \in S$, for $M=2N\mathbf{E}t_e$,
\begin{equation}\label{eqn: Ptailbound}
\begin{split}
&\sum_{\lambda = M \text{diam}(S)}^\infty \lambda \max_{y \in S} \mathbf{P}(T(x_0,y) \geq \lambda) \\
&\qquad\qquad\leq (2M \text{diam}(S) \mathbf{Var} t_e)^{2d} \sum_{\lambda=M \text{diam}(S)}^\infty \lambda^{1-4d} = C(\text{diam }S)^{2-2d}\ .
\end{split}
\end{equation}
The result now follows by subadditivity, \index{subadditivity}%
\begin{align*}
\mathbf{E}\left[ \max_{x,y \in S} T(x,y) \right] \leq 2 \mathbf{E} \left[ \max_{y \in S} T(x_0,y) \right]  &\leq 2M \text{diam }S \\
&\qquad+ 2 \#S\sum_{\lambda=M\text{diam}(S)}^\infty \max_{y \in S} \mathbf{P}(T(x_0,y) \geq \lambda) \\
&\leq C\text{diam } S\ ,
\end{align*}
where in the second step we have used \[\mathbf{P}(\max_{y \in S}T(x_0,y) \geq \lambda) \le (\#S) \max_{y \in S} \mathbf{P}(T(x_0,y) \geq \lambda),\] and in the final step we
have used \eqref{eqn: Ptailbound}.
\end{proof}

\section{General edge-weight distributions}
To prove Theorem \ref{thm: DHS13} for general edge weight distributions we must find an alternative to the log-Sobolev inequality that we used in the case of Bernoulli weights to obtain the bound
\[\sum_{i=1}^\infty Ent_i|\Delta_i F_n|^2 \le Cn.\]
The proof rests on two main ideas:
\begin{enumerate}
\item Representing each of the edge weights $t_e$ as the pushforward of a Bernoulli sequence $(\omega_{e,j})_{j\ge 1}$. This representation allows to use the discrete log-Sobolev inequality \eqref{eqn: logsobolev} to bound the entropy by a double sum over $e$ and $j$ of discrete derivatives corresponding to the two values of $\omega_{e,j}$. After a computation involving a variant of our ``Ingredient 1'' (see Lemma \ref{linear}), it is found that
\begin{equation}\label{eqn: logsum}
\sum_{i=1}^\infty Ent_i|\Delta_i F_n|^2  \le C\mathbf{E}\sum_{e\in G_n(0,n\mathbf{e}_1)} (1-\log F(t_e)),
\end{equation}
where $F$ is the distribution function of $\mu$. Note that without the $\log F$ term, the sum on the right would be $O(n)$ by \eqref{eqn: Gn-lin}.
\item A lattice animals argument to show that under the assumption \eqref{eqn: logte}, the quantity \eqref{eqn: logsum} is of order $O(n)$. \index{lattice animals}%
\end{enumerate}

Let $\omega_{e,j}$, $e\in \mathcal{E}(\mathbb{Z}^d)$, $j\ge 1$ be a collection of independent Bernouli($\frac{1}{2}$)-distributed random variables on a product probability space $(\Omega_B,\pi=\prod_{e,j}\pi_{e,j})$. We denote
\[
\omega_B = \left\{ \omega_{e,j} : e \in \mathcal{E}^d,~j \geq 1 \right\}\ .
\]
For fixed $e$, denote by $\omega_e$ the vector $(\omega_{e,j})_{j\ge 1}$, and define the random variable
\begin{equation}
\label{eqn: Uedef}
U_e(\omega_e)=\sum_{j=1}^\infty \frac{\omega_{e,j}}{2^j}.
\end{equation}
Under $\pi$, for each $e$, $U_e(\omega_e)$ is uniformly distributed on $[0,1]$, and the collection $\{U_e(\omega_e)\}_{e\in \mathcal{E}^d}$ is independent. 

We denote by $F(x)$, $x\ge 0$ the distribution function of $\mu$, and by $I$ the infimum of the support:
\[I=\inf\{x:F(x)>0\}.\] 
The right-continuous inverse of $F$ is
\[F^{-1}(y) = \inf\{x: F(x) \ge y\}.\]
If $U$ is uniformly distributed on $[0,1]$, we have 
\begin{equation}
\label{eqn: uniform}
\mathbb{P}(F^{-1}(U)\le x) = \mathbb{P}(U\le F(x))= F(x),
\end{equation}
that is, $F^{-1}(U)$ has distribution $\mu$. Letting
\begin{equation}\label{eqn: Tedef}
\varphi_e(\omega_e) = F^{-1}(U_e(\omega_e)),
\end{equation}
the distribution of $\varphi_e$ under $\pi$ is $\mu$.

We gather the $\varphi_e$ into a product map $\varphi:= \Omega_B \to \Omega=[0,\infty)^{\mathcal{E}^d}$:
\[
\varphi(\omega_B) = (\varphi_e(\omega_e) : e \in \mathcal{E}^d)\ .
\]
By \eqref{eqn: uniform}, $\pi \circ \varphi^{-1} = \mathbb{P}$.

Functions $f$ (in particular, $F_n$) on the original space $\Omega$ can be written as functions on $\Omega_B$, through the map $\varphi$. We will estimate discrete derivatives, so for a function $f: \Omega_B \to \mathbb{R}$, set
\begin{equation}
\left(\nabla_{e,j} f\right)(\omega_B) = f(\omega_B^{e,j,+}) - f(\omega_B^{e,j,-})\ ,
\end{equation}
where $\omega_B^{e,j,+}$ agrees with $\omega_B$ except possibly at $\omega_{e,j}$, where it is $1$, and $\omega_B^{e,j,-}$ agrees with $\omega_B$ except possibly at $\omega_{e,j}$, where it is 0.

\begin{lemma}\label{lem: bern}
We have the following inequality:
\[
\sum_{i=1}^\infty Ent(\Delta_i F_n ^2) \leq \sum_{e,j} \mathbf{E}_\pi\left(\nabla_{e,j}F_n\circ \varphi \right)^2\ .
\]
\end{lemma}
\begin{proof}
Once the notation is properly set up, the proof is a straightforward application of the log-Sobolev inequality and tensorization of entropy like what we saw in the case of $\mu$ Bernoulli. 

We let $G= F_n \circ \varphi$, and denote by $\mathcal{G}_k$ the sigma-algebra generated by $\{\omega_{e_r,j}, r\le k, j\in \mathbb{N}\}$. Then we have
\[\mathbf{E}[F_m \mid \mathcal{F}_i](\varphi(\omega_B))=\mathbf{E}[G\mid \mathcal{G}_i](\omega_B)\]
$\pi$-almost surely. Define
\[W_i = \mathbf{E}_\pi[ G\mid \mathcal{G}_i]-\mathbf{E}_\pi[ G\mid \mathcal{G}_{i-1}] \]
Using Propostion \ref{prop: tensorization}, we find
\[\sum_{i=1}Ent_\mu (\Delta_i F_n^2) = \sum_{i=1}Ent_\pi (W_i^2)
\le \sum_{k=1}^\infty \mathbf{E}_\pi \sum_{e,j} Ent_{\pi_{e,j}} (W_i)^2.\]

Using the log-Sobolev inequality \eqref{eqn: logsobolev}, the last quantity is bounded by
\[\sum_{i=1}^\infty \mathbf{E}_\pi \sum_{e,j} Ent_{\pi_{e,j}} (W_i)^2 \le C\sum_{i=1}^\infty \sum_{e,j} \mathbf{E}_\pi[(\nabla_{e,j} W_i)^2]. \]
Computing $\nabla_{e,j}W_k$ as in \eqref{eqn: marta} and using orthogonality of increments, we can sum over $k$ to find
\begin{equation}\label{eqn: Gs}
C\sum_{e,j} \mathbf{E}_\pi(\nabla_{e,j} G)^2.
\end{equation}
The lemma is proved.
\end{proof}

\subsection{The key estimate}
The $L^2$ norm of the discrete derivatives appearing in \eqref{eqn: Gs} can be expressed in terms of the variables $U(\omega_e)$ in \eqref{eqn: Uedef} as
\begin{align*}
&\mathbf{E}_\pi[(\nabla_{e,j}G)^2] = \mathbf{E}_\pi (F_n(1,(\omega_{e',i})_{(e',i)\neq(e,j)})-F_n(0,(\omega_{e',i})_{(e',i)\neq(e,j)}))^2\\
&\qquad= \frac{1}{2}\mathbf{E}_\pi[(F_n(\omega_{e,j}, (\omega_{e',i})_{(e',i)\neq(e,j)})-F_n(0, (\omega_{e',i})_{(e',i)\neq(e,j)}))^2\mathbf{1}_{\{\omega_{e,j}=1\}}]\\
&\qquad\le \frac{1}{2}\mathbf{E}_\pi[(F_n(U_e(\omega_e),(U_{e'})_{e'\neq e})-F_n(U_e(\omega_e-2^{-j}),(U_{e'})_{e'\neq e}))^2\mathbf{1}_{\{U_e\ge 2^{-j}\}}],
\end{align*}
with the obvious identifications of $F_n$ with $F_n\circ \varphi$.
Recalling the definition of $F_n$ \eqref{eqn: Fn}, the last quantity is bounded by
\begin{equation}
\begin{split}
&\mathbf{E}[(F_n(U_e(\omega_e),(U_{e'})_{e'\neq e})-F_n(U_e(\omega_e-2^{-j}),(U_{e'})_{e'\neq e}))^2\mathbf{1}_{\{U_e\ge 2^{-j}\}}]\\
&\qquad\le \frac{1}{\# B_m}\sum_{z\in B_m} \mathbf{E}_\pi[(T_z(U_e(\omega_e))-T_z(U_e(\omega_e)-2^{-j}))^2\mathbf{1}_{\{U_e\ge 2^{-j}\}}].
\end{split}
\end{equation}
For simplicity of notation, we have omitted the dependence of $T_z:=T(z,z+n\mathbf{e}_1)$ on variables $U_{e'}$, $e'\neq e$. We have so far shown that to estimate
\[\sum_{i=1}Ent_\mu (\Delta_i F_n^2),\]
it will suffice to estimate
\begin{equation}\label{eqn: overj}
\sum_{e\in\mathcal{E}(\mathbb{Z}^d)} \frac{1}{\# B_m}\sum_{z\in B_m} \sum_{j=1}^\infty \mathbf{E}_\pi[(T_z(U_e(\omega_e))-T_z(U_e(\omega_e)-2^{-j}))^2_+\mathbf{1}_{\{U_e\ge 2^{-j}\}}].
\end{equation}
We would like to perform the summation over $j$ and obtain a term similar to the summands in \eqref{eqn: logsum}.

To move further, we need an estimate for
\begin{equation}\label{eqn: Tz}
T_z(U_e(\omega_e))-T_z(U_e(\omega_e)-2^{-j})
\end{equation}
that is summable in $j$.
For this purpose, we introduce a more refined version of ``Ingredient 1'' (See section \ref{sec: ingredients}), concerning the effect on $T_n$ of changing one edge weight. Whereas previously Lemma \ref{linear} was sufficient, we will need the following more precise result:
\begin{lemma}\label{lem: linearb}
The random variable
\[D_{z,e_i}=\sup[\{r\ge 0: t_e=r \text{ and } e \text{ is in a geodesic from } z \text{ to } z+z\}\cup \{0\}]\]
is almost surely finite. If $0\le s \le t$,
\begin{equation}
T_z(t,t_{e_i^c})-T_z(s,t_{e_i^c})=\min\{t-s,(D_{z,e}-s)_+\}.
\end{equation}
\end{lemma}
The content of Lemma \ref{lem: linearb} is clear: the difference in passage times $T_z(t)-T_z(s)$ is linear, as long as $t$ is not so large that the edge $e_i$ is no longer in a geodesic. The logic of the proof is similar to the argument in Lemma \ref{linear}, see \cite{DHS13}. Two crucial properties of $D_{z,e}$ are 
\begin{enumerate}
\item $D_{z,e}$ depends only on edge weights $t_{e'}$, $e'\neq e$, but not on $t_e$.
\item  $T_z(U_e(\omega_e))$ is constant on $U_e\in [F(D_{z,e}^-),1]$ (for fixed values of $t_{e^c}$.
\end{enumerate}
\index{Rossignol, Raphael}%
To perform the summation over $j$ in \eqref{eqn: overj}, we combine Lemma \ref{lem: linearb} with the following lemma, due to R. Rossignol, which summarizes a less transparent computation which appeared in \cite{DHS13}. 
\begin{lemma} \label{lem: uniformvar}
Let $a,\tau\in [0,1]$. Suppose f is nonnegative, non-decreasing on $[0,1]$, and that $f$ is constant on $[a,1]$. If $\tau \le 1/2$, then 
\begin{equation}\label{eqn: a>1/2}
\int_\tau^1 (f(x)-f(x-\tau))^2\,\mathrm{d}x \le \int_0^1 f^2(x) \mathbf{1}_{\{x\ge 1-\tau\}}\,\mathrm{d}x.
\end{equation}
Moreover:
\begin{enumerate}
\item If $a\le \tau \le \frac{1}{2}$,
\begin{equation} \label{eqn: a<tau}
\int_\tau^1(f(x)-f(x-\tau))^2\,\mathrm{d}x\le 2a\int_0^1 f^2(x)\,\mathrm{d}x.
\end{equation}
\item If $\tau\le a\le \frac{1}{2}$,
\begin{equation}
\label{eqn: a>tau}
\int_\tau^1 (f(x)-f(x-\tau))^2\,\mathrm{d}x \le 2\tau \int_0^1 f^2(x)\,\mathrm{d}x.
\end{equation}
\end{enumerate}
\end{lemma}
We defer the proof to section \ref{sec: rossignol} below.

With Lemmas \ref{lem: linearb} and \ref{lem: uniformvar} in hand, we can now obtain \eqref{eqn: logsum}. Fix an edge $e$. We apply Lemma \ref{lem: uniformvar} with
\begin{align*}
f(x)&=T_z(x)-T_z(0),\\
\tau&=2^{-j},\\
a&=F(D_{z,e}^-).
\end{align*}
Writing the expectation over $U_e$ explicitly:
\begin{align*}
&\mathbf{E}_{\pi_e}[(T_z(U_e(\omega_e))-T_z(U_e(\omega_e)-2^{-j}))^2\mathbf{1}_{\{U_e\ge 2^{-j}\}}]\\
&\qquad=\int_{2^{-j}}^1 (T_z(x)-T_z(x-2^{-j}))^2\,\mathrm{d}x.
\end{align*}
For $j$ such that $F(D_{z,e}^-)<2^{-j}$ ($a<\tau$), we use \eqref{eqn: a<tau}:
\[\int_{2^{-j}}^1 (T_z(x)-T_z(x-2^{-j}))^2\,\mathrm{d}x \le  2F(D_{z,e}^-)\int_0^1(T_z(x)-T_z(0))^2\,\mathrm{d}x.\]
Recall that $T_z(x)$ also depends on $\omega_{e'}$, $e'\neq e$. By Lemma \ref{lem: linearb}, we have
\[0\le T_z(x)-T_z(0)\le \min\{D_{z,e}, \varphi_e(\omega_e)\},\]
where $\varphi_e$ was defined in \eqref{eqn: Tedef}, so 
\[\int_0^1(T_z(x)-T_z(0))^2\,\mathrm{d}x \le \mathbf{E}_{\pi_e}[\min\{D_{z,e}, \varphi_e(\omega_e)\}]^2.  \]
For  $j$ such that $2^{-j}\le F(D^-_{z,e})\le \frac{1}{2}$, $(\tau \le a\le \frac{1}{2})$, we use \eqref{eqn: a>tau}:
\[\int_{2^{-j}}^1 (T_z(x)-T_z(x-2^{-j}))^2\,\mathrm{d}x \le 2\cdot 2^{-j} \mathbf{E}_{\pi_e}[\min\{D_{z,e}, \varphi_e(\omega_e)\}]^2.\] 
We now sum over $j$ in \eqref{eqn: overj}. For $F(D_{z,e}^-)\le \frac{1}{2}$ we find
\begin{align*}
&\sum_{j=1}^\infty \mathbf{E}_{\pi_e}\Bigl[(T_z(U_e(\omega_e)-T_z(U_e(\omega_e)-2^{-j}))^2\mathbf{1}_{\{U_e\ge 2^{-j}\}}\Bigr] \\
&\qquad\le \!\!\sum_{j: 2^{-j}>F(D_{z,e}^-)} \!\!\!\!\!\!\mathbf{E}_{\pi_e}\Bigl[\min\{D_{z,e}, \varphi_e(\omega_e)\}\Bigr]^2+\!\!\!\sum_{j: 2^{-j}\le F(D_{z,e}^-)} \!\!\!\!\!\!\!\!\!2^{1-j}\mathbf{E}_{\pi_e}\Bigl[\min\{D_{z,e}, \varphi_e(\omega_e)\}\Bigr]^2\\
&\qquad\le  2F(D_{z,e}^-)(2-\log_2F(D_{z,e}^-))\mathbf{E}_{\pi_e}\Bigl[\min\{D_{z,e}, \varphi_e(\omega_e)\}\Bigr]^2.
\end{align*}

If instead $F(D_{z,e}^-)\ge \frac{1}{2}$, we use \eqref{eqn: a>1/2} in the sum over $j$ such that $2^{-j}\le F(D^-_{z,e})$:
\begin{align*}
&\sum_{j: 2^{-j}\le F(D_{z,e}^-)} \int_0^1 [(T_z(x)-T_z(x-2^{-j}))^2\mathbf{1}_{\{x\ge 2^{-j}\}}\,\mathrm{d}x\\
&\qquad\qquad\le \int_0^1 \min\{D_{z,e}, F^{-1}(x)\}^2\sum_{j: 2^{-j}\le F(D_{z,e}^-)}\mathbf{1}_{\{1-x\le 2^{-j}\}}(x)\,\mathrm{d}x\\
&\qquad\qquad\le \int_0^1  \min\{D_{z,e}, F^{-1}(x)\}^2 \log_2\frac{1}{1-x}\,\mathrm{d}x\\
&\qquad\qquad\le 2F(D^-_{z,e})\int_0^1  \min\{D_{z,e}, F^{-1}(x)\}^2  \log_2\frac{1}{1-x}\,\mathrm{d}x.
\end{align*}

Combining the previous calculations, we find:
\begin{equation}\label{eqn: ross}
\begin{split}
&\sum_{j=1}^\infty \mathbf{E}_\pi[T_z(U_e(\omega_e)-T_z(U_e(\omega_e)-2^{-j})^2_+\mathbf{1}_{U_e\ge 2^{-j}\}}] \\
&\qquad\le C\mathbf{E}_{\pi_e}\,F^{-1}(U_e)^2\left(1+\log_2^+\frac{1}{1-U_e}\right) \mathbf{E}_\pi F(D_{z,e}^-)(1-\log_2 F(D_{z,e}^-).
\end{split}
\end{equation}
Note the use of independence of $D_{z,e}$ from $\omega_e$. To obtain an expression of the form $\eqref{eqn: logsum}$, we integrate by parts:
\begin{equation}
\begin{split}
-F(y^-)\log F(y^{-}) &= -\int\mathbf{1}_{[I,y)}(x)\log F(y^-)\,\mu(\mathrm{d}x)\\
&\le -\int \mathbf{1}_{[I,y)}(x)\log F(x)\,\mu(\mathrm{d}x).
\end{split}
\end{equation}
Thus
\[F(D_{z,e}^-)(1-\log_2 F(D_{z,e}^-) = \int_I^{D_{z,e}}(1-\log F(t_e)\,\mu(\mathrm{d}t_e).\]
Finally, the quantity
\[\mathbf{E}_{\pi_e}\,F^{-1}(U_e)^2\left(1+\log_2\frac{1}{1-U_e}\right)\]
can be estimated by
\[2Ent_\mu t_e^2 + 2\mathbf{E}_\mu t_e^2 \mathbf{E}_{\pi_e} \frac{1}{(1-U_e)^{1/2}} \]
using the variational charaterization of entropy, \eqref{eqn: entropyvar}. By \eqref{eqn: logte}, this is a finite constant.

By Lemma \ref{lem: linearb}, if $t_e< D_{z_e}$, then $e\in G_n(z,z+n\mathbf{e}_1)$, so
\[\int_I^{D_{z,e}} \sum_{e\in \mathcal{E}(\mathbb{Z}^d} (1-\log F(t_e)\,\mu(\mathrm{d}t_e) \le \sum_{e\in G_n(z,z+n\mathbf{e}_1)} (1-\log F(t_e))\,\mu(\mathrm{d}t_e) \]
almost surely. Thus \eqref{eqn: ross} gives
\begin{align}
&\sum_{j=1}^\infty \mathbf{E}_\pi[(T_z(U_e(\omega_e)-T_z(U_e(\omega_e)-2^{-j}))^2_+\mathbf{1}_{\{U_e\ge 2^{-j}\}}]\\
&\qquad\le \mathbf{E}_\pi \sum_{e\in G_n(z,z+n\mathbf{e}_1)} (1-\log F(t_e))\,\mu(\mathrm{d}t_e).
\end{align}

Shifting by $z$, we have now derived the key estimate in \cite{DHS13}:
\begin{proposition}
Define $F_n$ by \eqref{eqn: Fn} and $\Delta_i F_n$ as in \eqref{eqn: Di}. Then there is a constant $C>0$ such that:
\[\sum_{i=1}^\infty Ent_i|\Delta_i F_n|^2  \le C\mathbf{E}\sum_{e\in G_n} (1-\log F(t_e)).\]
\end{proposition}

\subsection{Estimating the sum \eqref{eqn: logsum}}
To complete the proof of Theorem \ref{thm: DHS13}, we must show 
\[\mathbf{E}\sum_{e\in G_n} (1-\log F(t_e))\le Cn.\]
The sum in the expectation has the form
\[Y_n := \sum_{e\in G_n} w_e,\]
where the weights $w_e$ are given by
\[w_e=1-\log F(t_e).\]

\index{uniform (distribution)}%
Recall that if $F^{-1}$ is the right-continuous inverse of $F$ and $U$ is uniform on $[0,1]$, then $F^{-1}(U)$ is distributed like $t_e$ :
\begin{equation}\label{eqn: exp}
\mathbf{P}(w \ge r) = \mathbf{P}(F(F^{-1}(U))\le e^{1-r})\le \mathbf{P}(U\le e^{1-r})=e^{1-r}.
\end{equation}
Thus, the weights $w_e$ have exponential tails. For a lattice path $\gamma$, we define
\[N(\gamma)= \sum_{e\in \gamma} N(\gamma),\]
and
\begin{equation}
N_n:= \max_{\substack{\gamma: \#\gamma=n\\ 0\in \gamma}} N_n(\gamma).
\end{equation}
Clearly,
\[\{N_n> \beta n\} \subset \bigcup_{\gamma:\#\gamma= n, 0\in \gamma} \{N(\gamma)>\beta n\}.\]
On the other hand, the number of \emph{lattice animals} (connected subsets of $\mathbb{Z}^d$) of size $\le n$ containing the origin is bounded by $e^{Cn}$ for some $n$, so
\[\mathbf{P}(N_n>\beta n) \le e^{Cn} \mathbf{P}(\sum_{i=1}^n w_i >\beta n).\]
Here $w_i$, $1\le i \le n$ are i.i.d. random variables with the common distribution of the $w_e$.
By \eqref{eqn: exp}, it is straightforward to show that 
\[\mathbf{P}(N_n\ge \beta n) \le e^{-\beta n/4}\]
for large $\beta$.  From this tail bound, we find
\begin{equation}
\mathbf{E}N_n^4 = 4\int_0^\infty x^3\mathbf{P}(N_n\ge x)\,\mathrm{d}x\le (\beta_0n)^4+4\int_{\beta_0n}^\infty x^3e^{-x/4}\,\mathrm{d}x \le Cn^4.
\end{equation}

We can now estimate the expectation
\begin{align*}
\mathbf{E}Y_n &= \sum_{j=1}^\infty \mathbf{E}Y_n \mathbf{1}_{2^{j-1}n\le \#G_n < 2^jn}\\
&\le \sum_{j=1}^\infty (\mathbf{E}N_{2^jn}^4)^{1/4}\mathbf{P}(2^{j-1}n\le \#G_n < 2^jn)^{3/4}\\
&\le Cn \sum_{j=1}^\infty 2^j \mathbf{P}(2^{j-1}n\le \# G_n < 2^jn)^{3/4}.
\end{align*}
we have used H\"older's inequality in the second step. Using H\"older (for sums) again on the final sum in $j$, we find
\begin{align*}
&\sum_{j=1}^\infty 2^{-j/2} 2^{3j/2} \mathbf{P}(2^{j-1}n\le \#G_n < 2^jn)^{3/4}\\
&\qquad\le C\left(\sum_{j=1}^\infty 2^{2j} \mathbf{P}(2^{j-1}n\le \#G_n < 2^jn)\right)^{3/4}.
\end{align*}
By \eqref{eqn: Gn-lin}, the final sum is bounded by a constant, and Theorem \ref{thm: DHS13} is proved.

\subsection{Proof of Lemma \ref{lem: uniformvar}}\label{sec: rossignol}
Here we give the derivation of Lemma \ref{lem: uniformvar}. Notice the crucial role of positive association.
\begin{proof}
If $\tau \le 1/2$, then
\begin{align*}
\int_\tau^1 (f(x)-f(x-\tau))^2\,\mathrm{d}x &\le \int_{\tau}^1 (f^2(x)-f^2(x-\tau))\,\mathrm{d}x
\end{align*}
where we have used $f\ge 0$.
Next, 
\begin{align*}
 \int_\tau^1 f^2(x)-f^2(x-\tau)\,\mathrm{d}x &= \int_\tau^1 f^2(x)\, \mathrm{d}x - \int_0^{1-\tau}f^2(x)\,\mathrm{d}x\\
&\le \int_{1-\tau}^1 f^2(x)\,\mathrm{d}x\\
&= \int_0^1 f^2(x) \mathbf{1}_{\{x\ge 1-\tau\}}\,\mathrm{d}x.
\end{align*}
We have used the non-negativity of $f^2(x)$ to drop the integral over $[0,\tau]$. This shows \eqref{eqn: a>1/2}.

If $a\le \tau\le \frac{1}{2}$, since $f$ is constant over $[a,1]$,
\[\int_\tau^1(f(x)-f(x-\tau))^2\,\mathrm{d}x = \int_\tau^{a+\tau}(f(x)-f(x-\tau))^2\,\mathrm{d}x.\]
By monotonicity, the right side is no bigger than
\[\int_{\tau}^{a+\tau} f^2(x)\,\mathrm{d}x = \int_{\tau}^1 f^2(x)\mathbf{1}_{\{x\le a+\tau\}}\,\mathrm{d}x.\]
The Chebyshev association inequality \cite[Theorem~2.14]{BLM} now gives \eqref{eqn: a<tau}:
\begin{align*}
\int_{\tau}^1 f^2(x)\mathbf{1}_{\{x\le a+\tau\}}\,\mathrm{d}x &\le \frac{1}{1-\tau}\int^1_\tau f^2(x)\,\mathrm{d}x \int_\tau^1 \mathbf{1}_{\{x\le a+\tau\}}\,\mathrm{d}x\\
&= \frac{a}{1-\tau}\int_\tau^1 f^2(x)\,\mathrm{d}x\\
&\le 2a\int_0^1 f^2(x)\,\mathrm{d}x.
\end{align*}

When $\tau \le a\le \frac{1}{2}$, we again have
\[\int_\tau^1 (f(x)-f(x-\tau))^2\,\mathrm{d}x=\int_\tau^{a+\tau} (f(x)-f(x-\tau))^2\,\mathrm{d}x.\]
Expanding the square and using monotonicity as in the case $\tau \ge 1/2$, we find
\begin{align*}
\int_\tau^{a+\tau} (f(x)-f(x-\tau))^2\,\mathrm{d}x &\le \int_\tau^{a+\tau}f^2(x)\,\mathrm{d}x-\int_0^af^2(x)\,\mathrm{d}x\\
&= \int_a^{a+\tau} f^2(x)\,\mathrm{d}x - \int_0^\tau f^2(x)\,\mathrm{d}x\\
&\le \int_a^1 f^2(x)\mathbf{1}_{\{x\le a+\tau\}}\,\mathrm{d}x\\
&\le \frac{\tau}{1-a}\int_0^1 f^2(x)\,\mathrm{d}x \\
&\le 2\tau \int_0^1f^2(x)\,\mathrm{d}x.
\end{align*}
In the second-to-last step, we have used Chebyshev's association inequality. 
\end{proof}

\section{Concentration}\index{concentration}%
We remark briefly on how to derive the concentration result \ref{thm: concentration}. The proof is relies on the following (see \cite{BLM}).
\begin{proposition}\label{prop: varp}
Suppose $Z\ge 0$ is a random variable such that there exist constants $0<C\le B$ such that
\[\mathbf{Var} e^{tZ/2} \le Ct^2\mathbf{E}e^{tZ}<\infty \]
for $t\in (0,B^{-1/2})$, then
\[\psi_Z(t):=\mathbf{E}[e^{tZ}] \le -2\log(1-Ct^2), \quad  t\in (0,B^{-1/2}).\] In particular, we have exponential concentration:
\[\mathbf{P}(Z\ge \lambda) \le e^{-t \lambda}\mathbf{E}e^{tZ}\le e^{-t \lambda} \frac{1}{(1-Ct^2)^2}.\]
\end{proposition}
Proposition \ref{prop: varp} reduces Theorem \ref{thm: concentration} to showing
\begin{equation}\label{eqn: K}
\mathbf{Var} (e^{\lambda F_n /2})\le K\lambda^2 \mathbf{E}e^{\lambda F_n}, \quad |\lambda|<\frac{1}{2\sqrt{K}}.
\end{equation}
with $K=\frac{Cn}{\log n}$. The proof of \eqref{eqn: K} follows a similar strategy to the bound for $\mathbf{Var} F_n $ presented in the previous part of this article, starting from inequality \eqref{eqn: FS} with $f=e^{\lambda F_n}$. The derivation of the key estimate for the discrete derivatives is somewhat more involved. We refer to \cite{DHS14} for details.

Finally, we note the following result, which improves the exponential concentration result of Kesten \cite{kesten93} to Gaussian concentration. It was first derived by M. Talagrand. An alternate proof appears in \cite{DHS14}.
\begin{theorem}
Let $d \ge 2$. Assuming \eqref{eqn: zeroprob} and $\mathbf{E}e^{\alpha t_e} <\infty$ for some $\alpha>0$, there exist $c,C>0$ such that
\[\mathbf{P}(T_n-\mathbf{E}T_n \ge t\sqrt{n})\le e^{-ct^2} \quad t\in(0,C\sqrt{n}).\]
If $\mathbf{E}Y^2<\infty$, where $Y$ is the minimum of $2d$ i.i.d. copies of $t_e$, then also
\[\mathbf{P}(T_n - \mathbf{E}T_n\le -t\sqrt{n})\le e^{-ct^2},\]
for all $t\ge  0$.
\end{theorem}

\bibliography{Sosoe}

\providecommand{\bysame}{\leavevmode\hbox to3em{\hrulefill}\thinspace}
\providecommand{\MR}{\relax\ifhmode\unskip\space\fi MR }
\providecommand{\MRhref}[2]{%
  \href{http://www.ams.org/mathscinet-getitem?mr=#1}{#2}
}
\providecommand{\href}[2]{#2}
\begin{thebibliography}{10}

\bibitem{50years}
Antonio Auffinger, Jack Hanson, and Michael Damron, \emph{50 years of first
  passage percolation}, University Lecture Series, vol.~68, American
  Mathematical Society, Providence, RI, 2017.

\bibitem{bakryanalysis}
Dominique Bakry, Ivan Gentil, and Michel Ledoux, \emph{Analysis and geometry of
  markov diffusion operators}, vol. 348, Springer Science \& Business Media,
  2013.

\bibitem{BR}
Michel Benaim and Rapha{\"e}l Rossignol, \emph{Exponential concentration for
  first passage percolation through modified poincar{\'e} inequalities},
  \textbf{44} (2008), no.~3, 544--573.

\bibitem{BKS}
Itai Benjamini, Gil Kalai, and Oded Schramm, \emph{First passage percolation
  has sublinear distance variance}, Ann. Probab. \textbf{31} (2003), no.~4,
  1970--1978.

\bibitem{bonami}
Aline Bonami, \emph{{\'E}tude des coefficients de fourier des fonctions de
  $l^p(g)$},  \textbf{20} (1970), no.~2, 335--402.

\bibitem{BLM}
S.~Boucheron, G.~Lugosi, and P.~Massart, \emph{Concentration inequalities: A
  nonasymptotic theory of independence}, OUP Oxford, 2013.

\bibitem{ch:Corwin}
Ivan Corwin, \emph{Exactly solving the kpz equation}, Proc. Amer. Math. Soc.
  (2018), To appear.

\bibitem{ch:Damron}
Michael Damron, \emph{Random growth models: shape and convergence rate}, Proc.
  Amer. Math. Soc. (2018), To appear.

\bibitem{DHS13}
Michael Damron, Jack Hanson, and Philippe Sosoe, \emph{Sublinear variance in
  first-passage percolation for general distributions}, Probability Theory and
  Related Fields \textbf{163} (2015), no.~1-2, 223--258.

\bibitem{DHS14}
Michael Damron, Jack Hanson, Philippe Sosoe, et~al., \emph{Subdiffusive
  concentration in first-passage percolation}, Electron. J. Probab \textbf{19}
  (2014), no.~109, 1--27.

\bibitem{FS}
Dvir Falik and Alex Samorodnitsky, \emph{Edge-isoperimetric inequalities and
  influences}, Combinatorics, Probability and Computing \textbf{16} (2007),
  no.~05, 693--712.

\bibitem{garban}
Christophe Garban and J~Steif, \emph{Lectures on noise sensitivity and
  percolation}.

\bibitem{gross}
Leonard Gross, \emph{Logarithmic sobolev inequalities}, American Journal of
  Mathematics \textbf{97} (1975), no.~4, 1061--1083.

\bibitem{HH}
David~A. Huse and Christopher~L. Henley, \emph{Pinning and roughening of domain
  walls in ising systems due to random impurities}, Phys. Rev. Lett.
  \textbf{54} (1985), 2708--2711.

\bibitem{HHF}
David~A. Huse, Christopher~L. Henley, and Daniel~S. Fisher, \emph{Huse, henley,
  and fisher respond}, Phys. Rev. Lett. \textbf{55} (1985), 2924--2924.

\bibitem{johansson2000}
Kurt Johansson, \emph{Shape fluctuations and random matrices}, Communications
  in mathematical physics \textbf{209} (2000), no.~2, 437--476.

\bibitem{KPZ}
Mehran Kardar, Giorgio Parisi, and Yi-Cheng Zhang, \emph{Dynamic scaling of
  growing interfaces}, Phys. Rev. Lett. \textbf{56} (1986), 889--892.

\bibitem{KZ}
Mehran Kardar and Yi-Cheng Zhang, \emph{Scaling of directed polymers in random
  media}, Phys. Rev. Lett. \textbf{58} (1987), 2087--2090.

\bibitem{kestenaspects}
Harry Kesten, \emph{Aspects of first passage percolation}, {\'E}cole
  d'{\'E}t{\'e} de Probabilit{\'e}s de Saint Flour XIV-1984.

\bibitem{kesten93}
\bysame, \emph{On the speed of convergence in first-passage percolation}, The
  Annals of Applied Probability (1993), 296--338.

\bibitem{KK}
J.~M. Kim and J.~M. Kosterlitz, \emph{Kim and kosterlitz reply}, Phys. Rev.
  Lett. \textbf{63} (1989), 1192--1192.

\bibitem{MW}
Jin~Min Kim and Sang-Woo Kim, \emph{Restricted solid-on-solid model with a
  proper restriction parameter $n$ in $4+1$ dimensions}, Phys. Rev. E
  \textbf{88} (2013), 034102.

\bibitem{newmanpiza}
Charles~M Newman and Marcelo~ST Piza, \emph{Divergence of shape fluctuations in
  two dimensions}, The Annals of Probability (1995), 977--1005.

\bibitem{ch:Seppalainen}
Timo Sepp\"al\"ainen, \emph{The corner growth model with exponential weights},
  Proc. Amer. Math. Soc. (2018), To appear.

\bibitem{talagrand-russo}
Michel Talagrand, \emph{On russo's approximate zero-one law}, The Annals of
  Probability (1994), 1576--1587.

\end{thebibliography}
\bibliographystyle{amsplain}

\end{document}